\documentclass[11pt]{amsart}
\pdfoutput=1

\usepackage{amsfonts, amstext, amsmath, amsthm, amscd, amssymb}
\usepackage{epsfig, graphics, psfrag, overpic, color}
\usepackage{enumerate}

\newtheorem{thm}{Theorem}[section]
\newtheorem{lem}[thm]{Lemma}

\newtheorem{cor}[thm]{Corollary}

\newtheorem{defn}[thm]{Definition}

\newtheorem*{mainthm}{Theorem \ref{thm:pattern1}}
\newtheorem*{maincor}{Corollary \ref{cor:Wd}}

\def \Z{\mathbb{Z}}
\def \T{\mathcal{T}}
\def \K{\mathbb{K}}

\def \C{\mathcal{C}}

\begin{document}

\title[Knots without cosmetic crossings]{Knots without cosmetic crossings}
\author[C. J. Balm and E. Kalfagianni]{Cheryl Jaeger Balm and Efstratia Kalfagianni}

\address[]{Department of Mathematics, Kansas State
University, Manhattan, KS 66502}

\email[]{cjbalm@math.ksu.edu}

\address[]{Department of Mathematics, Michigan State
University, E Lansing, MI 48824} \ \ \
\email[]{kalfagia@math.msu.edu}

\thanks{ {Research supported by NSF grant DMS-1105843.
}}

\begin{abstract}  Let $K'$ be a knot that admits no cosmetic  crossing changes and let $C$ be a prime, non-cable
knot.
Then any knot that is a satellite of $C$ with winding number zero and pattern $K'$
admits no  cosmetic crossing changes.
 As a consequence  we 
prove the nugatory crossing conjecture for
Whitehead doubles of prime, non-cable knots.
\smallskip
\smallskip
\smallskip
\smallskip
\smallskip
\smallskip

\noindent {\it Keywords:} Companion torus, cosmetic crossing, pattern,  satellite.
\smallskip
\smallskip
\smallskip
\smallskip

\noindent {\it Mathematics Subject Classification:} 57M25, 57M27, 57M50
\end{abstract}

\maketitle

\bigskip


\section{Introduction}

 A \emph{crossing disk} for an oriented knot $K\subset S^3$ is an embedded disk $D\subset S^3$ such that $K$ intersects ${\rm int}(D)$ twice with zero algebraic intersection number.  A crossing change on $K$ can be achieved by performing $(\pm 1)$-Dehn surgery of $S^3$ along the \emph{crossing circle} $L = \partial D$.  More broadly,  a  \emph{generalized crossing change} of order $q \in \Z - \{ 0 \}$
 is achieved by $(-1/q)$-Dehn surgery along the crossing circle $L$ and results in introducing $q$ full twists to $K$ at the \emph{crossing disk} $D$ bounded by $L$.  
See Figure \ref{fig:nug}.
 A (generalized) crossing change of $K$ and its corresponding crossing circle $L$ are called \emph{nugatory} if $L$ bounds an embedded disk in $S^3 - \eta(K)$, where $\eta(K)$ denotes a regular neighborhood of $K$ in $S^3$.  Obviously, a generalized crossing change of any order at a nugatory crossing of $K$ yields a knot isotopic to $K$.  
 \begin{defn}{\rm
A (generalized) crossing change on $K$ and its corresponding crossing circle are called \emph{cosmetic} if the crossing change yields a knot isotopic to $K$ and is performed at a crossing of $K$ which is \emph{not} nugatory.}
\end{defn}

%

It is a fundamental  open question whether there exist knots that admit cosmetic crossing changes \cite[Problem 1.58]{kirby}.
This question, often referred to as the \emph{nugatory crossing conjecture},  has been answered in the negative for many classes of knots.  It follows from work of Scharlemann and Thompson \cite{schar-thom}, based on techniques  of Gabai \cite{gabai},  that the unknot admits no  cosmetic generalized crossing changes. Torisu
showed that   2-bridge knots  admit no cosmetic generalized crossing changes    \cite{torisu},
and  Kalfagianni showed that the same is true for fibered knots \cite{kalf}. Obstructions to cosmetic crossing changes in genus-one knots were found by the authors with Friedl and Powell in \cite{balm}, where it is shown that genus-one, algebraically non-slice knots admit no cosmetic generalized crossing changes. The objective of the current paper is to study the behavior of potential cosmetic crossing changes under the operation of forming satellites with winding number zero.

To state our results, let $\K$ denote the class of knots which do not  admit cosmetic generalized crossing changes.  By the previous paragraph, $\K$ contains all fibered knots, 2-bridge knots and genus-one, algebraically non-slice knots.  Torisu shows in \cite{torisu} that the connect sum of two or more knots in $\K$ is also in $\K$.

\begin{figure}
  \begin{center}
    \begin{tabular}{ccccc}

       \psfig{figure=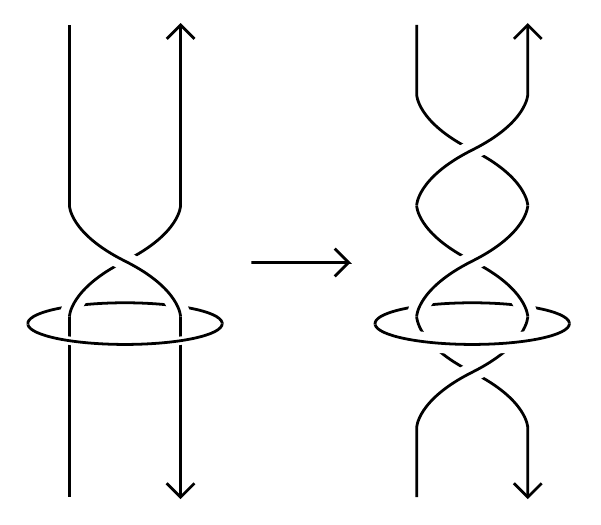,height=1.1in, clip=}& \hspace{.01in} &
       \psfig{figure=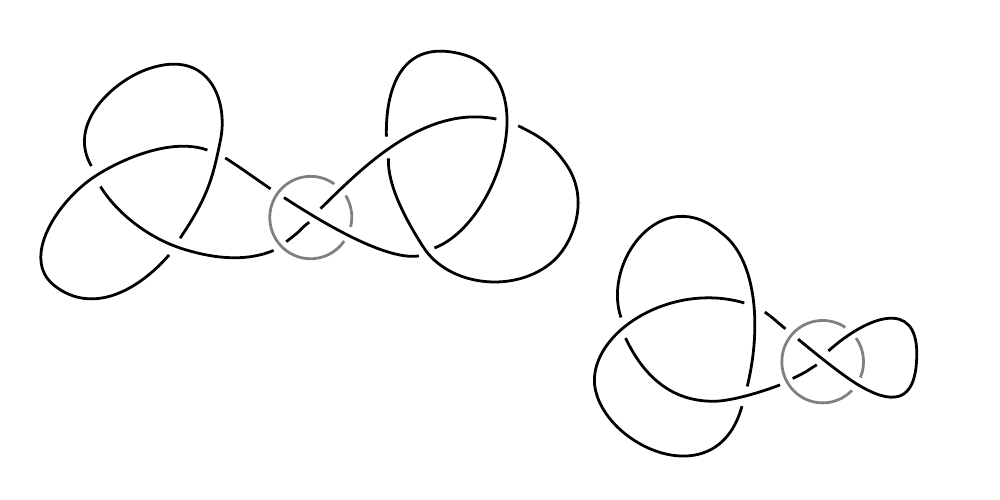,height=1.1in, clip=}
    \end{tabular}
  \end{center}
  \caption{Left:  A generalized crossing change of order 2.  Right:  Examples of nugatory crossing circles.}
  \label{fig:nug}
\end{figure}

\begin{defn}\label{def:wrap}{\rm We will say a torus $T$ is \emph{standardly embedded} in $S^3$ if $T$ bounds a solid torus on both sides. 

A knot $K$ embedded in a solid torus $V$, is called \emph{geometrically essential}
if the geometric intersection of $K$ with every meridian disc of $V$ is non-trivial.

 For a knot $K$ embedded in a solid torus $V$, the \emph{winding number}, $w(K,V)$, is the  algebraic intersection number of $K$ with a meridian disk of $V$.} 
\end{defn}

In \cite{balm2},  Balm showed that any prime satellite knot with pattern a non-satellite knot in  $\K$ does not admit cosmetic
generalized  crossing changes of order greater than five. Here we restrict ourselves to satellites with winding number zero, and we obtain the following stronger result.

\begin{thm}\label{thm:pattern1} Let $C$ be a non-trivial,  prime, non-cable knot and let $V'$ be a standardly embedded solid torus in $S^3$.  Let $K'\in \K$  and suppose that $K'$ is embedded in $V'$ so that it is geometrically essential and such that  $w(K',V')=0$.  Then any knot $K$ that is a satellite of $C$ with pattern $(V',  K')$  admits no  cosmetic generalized crossing changes of any order. That is $K$ satisfies the nugatory crossing conjecture.
\end{thm}

Theorem \ref{thm:pattern1} has the following corollary.

\begin{cor}\label{cor:Wd} Let $K$ be a prime, non-cable knot. 
Then no Whitehead double of $K$ admits a cosmetic generalized crossing change of any order.
\end{cor}

To outline the main ingredients of the proof of Theorem \ref{thm:pattern1}, let $K$ be a satellite knot that has zero winding number in a companion solid torus $V$ with pattern $K' \in \K$.
First we use results of Gabai \cite{gabai} to prove that a cosmetic crossing disk $D$ of $K$
can be isotoped to lie inside $V$. Using this and properties of toroidal decompositions of knot complements, we argue that an order-$q$
cosmetic crossing disk $D$ of $K$ gives an order-$q$ crossing change of $K'$, that yields a knot isotopic to $K'$.
Then, using a result of McCullough \cite{mccullough},  we show that the fact that $D$ is a cosmetic crossing disk for $K$
implies that $D$ is also cosmetic for $K'$. 
This, in turn, contradicts the assumption that $K' \in \K$ and proves Theorem \ref{thm:pattern1}.

The paper is organized as follows: In section two we study the interplay between cosmetic crossing disks and companion tori with zero
winding number, and we prove the auxiliary results needed for the proof of Theorem
\ref{thm:pattern1}. In section three we prove Theorem
\ref{thm:pattern1} and Corollary \ref{cor:Wd}.

\smallskip

\section{Crossing circles, companion tori and isotopies}\label{s:background}

We begin by recalling some definitions.

\begin{defn}{\rm Let $V'$ be a standardly embedded solid torus in $S^3$, and let $K'$ be a knot embedded in $V'$ so that $K'$ is geometrically essential in $V'$ and not
the core of $V'$.  A solid torus $V \subset S^3$ is \emph{knotted}  if the core of $V$ is not isotopic to the unknot in $S^3$.
 
Let  $f:(V', K') \to  S^3$ be an embedding such that the solid torus  $V=f(V')$ is knotted.  
A \emph{satellite knot} with \emph{pattern} $K'$ is the image $K = f(K')$.  If $C$ is the core of $V$, then $C$ is a \emph{companion knot}
of $K$, and we may call $K$ a \emph{satellite of $C$}.  The torus $T=\partial V$ is a \emph{companion torus} of $K$.  We may similarly define a \emph{satellite link} if $K'$ is a non-split link.

A \emph{cable} knot is a torus knot or a satellite knot with  pattern a torus knot.}
\end{defn}

Given a 3-manifold $N$ and a submanifold $F \subset N$ of co-dimension 1 or 2, $\eta (F)$ will denote a regular neighborhood of $F$ in $N$.  For a knot or link $K \subset S^3$, we define $M_K = \overline{S^3 - \eta(K)}$.  

Given a knot $K$,  let $K(q)$ denote the knot obtained via an order-$q$ generalized crossing change at a crossing circle $L$. We will also use the notation $K(0)$ when we wish to be clear that we are referring to the embedding of $K$ in $S^3$ before any crossing change occurs. Setting $M_{K \cup L}=\overline{ S^3- \eta(K\cup L)}$,
we will let $M(q)$ denote the 3-manifold obtained from $M_{K \cup L}$ via a Dehn filling of slope $(-1/q)$ along $\partial \eta (L)$.  So for $q \in \Z - \{ 0 \},\ M(q) = M_{K(q)}$ and $M(0) = M_K$.  

The first lemma which we will need in the proof of the results stated in the introduction is the following.

\begin{lem}\label{lem:LinV}
Let $K$ be a satellite knot, $T$ be a companion torus for $K$,  and $V$ be the solid torus bounded by $T$ in $S^3$.
Suppose that  $w(K, V) = 0$ and that there are no essential annuli in $\overline{S^3 - V}$.   If $D$  is a cosmetic crossing disk 
 for $K$ we can isotope it so that it lies in $V$.
\end{lem}

\begin{proof}
Let $K$ be as in the statement of the lemma, and suppose that $D$ is an order-$q$ cosmetic crossing disk
with $L=\partial D$.  Since $\text{lk}(K, L)=0$,
$K$ bounds s Seifert surface in the $\overline{S^3 - \eta(L)}$.
Let $S$ be a minimal genus  Seifert surface for $K$ in $\overline{S^3 - \eta(L)}$.
The intersection $S \cap D$  consists of a single embedded arc $\alpha$ and a collection of simple closed curves. Simple closed curves
that are parallel to $\partial D$ can be eliminated by sliding $S$ off the boundary of $D$, while simple closed curves that bound disks
in $D- \alpha$ can be eliminated by isotopy of $S$ in $\overline{S^3 - \eta(L)}$.
Hence we may choose  $S$ so that $S \cap D$ is a single embedded arc $\alpha$.  Then performing $(-1/q)$-surgery at $L$ twists both $K$ and $S$ and produces a surface $S(q) \subset M(q)$ which is a Seifert surface for $K(q)$.  

Note that if $M_{K \cup L}$ were reducible, then $M_{K \cup L}$ would contain a separating 2-sphere which does not bound a 3-ball $B \subset M_{K \cup L}$.  Then $L$ would lie in a 3-ball disjoint from $K$.  Hence $L$ would bound a disk in this 3-ball, which is in the complement of $K$, so $L$ would be nugatory.  Since $L$ is cosmetic by assumption, we may conclude that $M_{K \cup L}$ is, in fact, irreducible.  By a result of Gabai
\cite[Corollary 2.4]{gabai}
at least one of $S$, $S(q)$ is of minimal genus for $K$, $K(q)$, respectively, in $S^3$. However, $K$ and $K(q)$ are isotopic,
and  the genus of $S$ is equal to that of $S(q)$. Thus both  $S$ and $S(q)$ are minimal genus surfaces for their respective knots.

We may isotope $D$ to lie in a neighborhood of $S$ so that if
$S \subset V$, then we may arrange that $D \subset V$.  Assume that $S \not\subset V$, and let $\C = S \cap T$.  We may isotope $S$ so that $\C$ is a collection of simple closed curves which are essential in both $S$ and $T$.  Since $w(K,V)=0$, $\C$ must be homologically trivial in $T$, where each component of $\C$ is given the boundary  orientation  from $S\cap V$.  Hence $\C$ bounds a collection of annuli in $T$ which we will denote by $A_0$. 

Let $S_0 = S - (S \cap V)$.  Suppose that $\chi (S_0) < 0$, where $\chi( \cdot )$ denotes the Euler characteristic.  We may create $S^*$ from $S$ by replacing $S_0$ by $A_0$, isotoped slightly, if necessary, so that the components of $A_0$ become disjoint.  Then $S^*$ is a Seifert surface for $K$, and $\chi (S^*) > \chi (S)$ since $\chi (A_0) = 0$.  This contradicts the fact that $S$ is a minimal genus Seifert surface for $K$, so it must be that $\chi (S_0) \geq 0$.  Since $S_0$ contains no closed component, and no component of $\C$ bounds a disk in $S$, we conclude that $S_0$ consists of annuli.

By assumption, there are no essential annuli in $\overline{S^3 - V}$, so each component of $S_0$ must be boundary parallel in $\overline{S^3 - V}$.  Thus we can isotope $S_0$ so that $S \subset V$, and therefore $D$ can be isotoped into $V$ as well. \end{proof}

The hypotheses in the statement of Lemma \ref{lem:LinV} assure that a cosmetic crossing disk can be isotoped to lie inside a companion torus. In particular, such a crossing disk is disjoint from the corresponding companion torus.
In the next three lemmas we deal with cosmetic crossing circles that are disjoint from companion tori. The first two
lemmas examine how companion tori of a knot $K$, that are disjoint from a crossing circle $L$,
behave under crossing changes along $L$ that don't change the isotopy class of $K$.

\begin{lem}\label{lem:KinV2} Suppose that $L$ is a nugatory crossing circle of order $q$ for  a knot $K$
and let   $F$ be an essential torus in $M_{K\cup L}=\overline{ S^3 -  \eta(K \cup L)}$. If $F$ becomes compressible
in one of $M(q)$ or $M(0)$ then it becomes compressible in both of them.
\end{lem}
\begin{proof}
By the proof of Lemma \ref{lem:LinV}, we have  
minimal genus surface $S$ of $K$ that intersects the cosmetic  crossing disk $D$, bounded by $L$,  at a single arc $\alpha$. In fact, we may take
$D$ to lie a neighborhood of $\alpha$ contained in a neighborhood of $S$.  Furthermore, the surface $S(q)$, obtained  after the cosmetic crossing change, is
of minimal genus for $K(q)$.

 Let $F$ be an essential torus in $M_{K\cup L}$. Suppose that $F$ becomes compressible in
$M(0)=\overline{ S^3 -  \eta(K)}$; the argument in the case that  $F$ becomes compressible in
$M(q)=\overline{ S^3 -  \eta(K(q))}$ is completely analogous.

 Let $E$ be a compressing disk for $F$ in $M(0)$. We have $K\cap  E=\emptyset$. Hence the intersection $S\cap E$, if non-empty,  is a collection of simple
closed curves.
Since $S$ is minimal genus for $K$, 
it is incompressible in $M(0)$. Since $M(0)$ is irreducible, by an innermost argument we may
isotope $S$ in $M(0)$ so that
$E\cap S=\emptyset$. Since $L$ lies in a neigborhood of $S$ we may arrange so that
$L\cap E=\emptyset$. Thus $E$ will survive as a compressing disk of $F$ in any manifold obtained  Dehn filling along $L$.
In particular  $F$ will remain compressible in $ \overline{S^3- \eta(K(q))}$.

\end{proof}

The following lemma is proved by arguments similar to those in the proofs of Lemma \ref{lem:LinV}.

\begin{lem}{ \rm {\cite[Lemma 4.6] {kalf-lin}}}\label{lem:KinV}
Let $V \subset S^3$ be a knotted solid torus such that $K \subset \rm{int}(V)$ is a knot which is geometrically essential in $V$ and $K$ has a crossing disk $D$ with $D \subset \rm{int}(V)$.  If $K$ is isotopic to $K(q)$ in $S^3$, then $K(q)$ is also geometrically  essential in $V$.  Further, if $K$ is not the core of $V$, then $K(q)$ is also not the core of $V$.
\end{lem}

The next lemma discusses the interplay of nugatory crossing changes with satellite operations.

\begin{lem}\label{lem:nugasat} Let $V'\subset S^3$ be a standardly embedded solid torus.
 Let $K'$ be a  knot that is geometrically essential in $V'$ and let $D'$ be a crossing disk for $K'$ that lies in $V'$. Suppose that there is an orientation preserving homeomorphism
$h: V' \to V'$ that takes the preferred longitude of $V'$ to itself and such that $h(K'(q))=K'(0)$.  Then if  $L'=\partial D'$ is nugatory for $K'$ in $S^3$,  it is also
nugatory for $K'$ in $V'$. 
\end{lem}

\begin{proof} Note that the existence of  $h$ as in the statement above implies that $K'(q)$  and $K'$ are ambiently isotopic in $V'$. 
Suppose $L'$ is nugatory.  Then $L'$ bounds a crossing disk $D'$ in $V'$ and another disk $D''$ in the complement of $K'$.  We may assume $D' \cap D'' = L'$.  

If $D''  \subset V$ then there is nothing to prove. Otherwise, let $A_{V'} = D' \cup (D''  \cap V')$.  Now $A_{V'}$
contains a component  $(A, \partial A)\subset  (V', \partial V')$  that is a
a properly embedded  planar surface in $V'$, so each component of  $\partial A $ is a preferred longitude of $V'$, and we have $D'\subset A$. 

Since $h$ takes the preferred longitude of $V'$ to itself,
there is an orientation preserving homeomorphism $H: S^3\to S^3$ such that $H|_{V'}= h$.
Up to isotopy on $\partial V'$ we may assume that  for every component $C\subset D''  \cap \partial V'$, we have
$H(C)=C$. 
The sphere $\Sigma=D'\cup D'' $ defines (possibly trivial) connect sum decompositions of $K'(q)$ and $K'(0)$.  Since $H$ preserves each decomposition we may arrange so that $H(\Sigma)=\Sigma$.
 Then $H$ maps each component of  $\Sigma \cap V'$ to itself; in particular $H(A)=A$. By further isotopy on $A$ we may arrange so that
$H$ leaves invariant each of $D'$, $D''$ and $V'$.
To summarize, we can find   ambient isotopy $\{ H_{t}\}_{0\leq t\leq 1}$ of $S^3$ so that: (i) $H_0=H$; (ii) $H_t(K'(q))=K'(0)$, for 
all $0\leq t\leq 1$; (iii) $H_1(D')=D'$, $H_1(D'' )=D'' $ and $H_1(V')=V'$.
Thus, by replacing our original $h$ with $H_1|_{V'}$, we may  assume $h(A)=A$ and $A$ cuts $V'$ into two components, $V'_1$ and $V'_2$.  An example where $A$ is an annulus is shown in Figure \ref{fig:torus}.

\begin{figure}
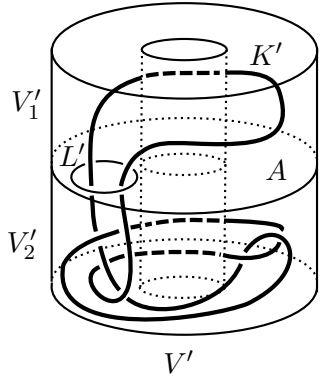
\caption{The solid torus $V'$, cut into two components by $A$, as in the proof of Lemma \ref{lem:nugasat}.}\label{fig:torus}
\end{figure}

\vskip 0.06in

{\it{ Case 1.} } Suppose that $ h( V'_i )= V'_i$, for $i=1,2$. 
Now $\Sigma= D' \cup D''  $ separates $S^3$ into two 3-balls, $B_1$ and $B_2$. Assume that  $V'_i \subset B_i$ for $i=1,2$.
On $A$ let $\tau^{-q}_{L'}$ denote the $-q$ power of the Dehn twist along $L'$.
If viewed as an auto-diffeomorphism of $\Sigma$, $\tau^{-q}_{L'}$ extends to an auto-diffeomorphism $T$of $B_2$
that is the identity off of a collar neighborhood of $B_2$. Define $f: S^3 \to S^3$ to be identity on $B_1$
and $T$ on $B_2$.
This is an orientation-preserving homeomorphism
that brings  $K'(q)$ to $K'(0)$.
Using $h$ as above we identify $V'_i$ with $h(V'_i)$ and $K'(0)$ with $K'(q)$ and simply denote it by $K'$.

Let $X$ be the 3-manifold obtained from $V'_2$ by drilling out a neighborhood of $K' \cap V'_2$. 
Now  $f$ restricted to $X$ 
 is a homeomorphism whose   restriction to $\partial V'_2$ is a Dehn twist along $L'$. 
  By a result of McCullough \cite[Theorem 1]{mccullough}, $L'$ bounds a disk in $V'_2 \subset (V' - \eta(K'))$.  Thus $L'$ is nugatory  in $V'$.

\vskip 0.06in

{\it{ Case 2.} } Suppose that $h: V' \to V'$ maps $ V'_2$ to $V'_1$. The sphere $\Sigma$ 
defines connect sum decompositions of $K'(q)$ and $K'(0)$. We have
$h(K'(q)\cap V'_1)=K'(0)\cap V'_2$ and $h(K'(q)\cap V'_2)=K'(0)\cap V'_1$.

On $A$ let $\tau^{-q}_{L'}$ denote the $-q$ power of the Dehn twist along $L'$ and let
$T$  its extension to $B_2$ as in Case 1. Define $f: S^3 \to S^3$ to be identity on $B_1$
and $T$ on $B_2$. Let $r: S^3\to S^3$ denote a rotation of 180 degrees with axis  a  circle  on $\Sigma$ passing through the two points comprising 
$\Sigma \cap K'(q)=\Sigma \cap K'(0)$. Up to isotopy in each of $B_1$ and $B_2$ we may assume that
$r(K'(q)\cap V'_1)=K'(0)\cap V'_2$ and $r(K'(q)\cap V'_2)=K'(0)\cap V'_1$.
Now define $g: S^3\to S^3$, by $g=f\circ r$. This is an orientation-preserving homeomorphism
with $g(K'(q)\cap V'_i)=K'(0)\cap V'_i$, for  $i=1,2$.
As in Case 1, using $h$  we identify $V'_i$ with $h(V'_i)$.
Now apply to $g$  an argument identical to this applied to $f$ above, and again appeal to McCullough's result to get the desired conclusion.

\end{proof}

\vskip 0.03in

 We close the section with the following lemma of Motegi  \cite{motegi} that we need for the proof of the main results.

\begin{lem}{\rm{ \cite[ Lemma 2.3]{motegi}}}\label{thm:motegi}
Let $K$ be a knot embedded in $S^3$ and let $V_1$ and $V_2$ be knotted solid tori in $S^3$ such that the embedding of $K$ is essential in $V_i$ for $i=1,2$.  Then there is an ambient isotopy $\phi: S^3 \to S^3$ leaving $K$ fixed such that one of the following holds.
\begin{enumerate}
\item $\partial V_1 \cap \phi (\partial V_2) = \emptyset$.
\item There exist meridian disks $D$ and $D'$ for both $V_1$ and $V_2$ such that some component of $V_1$ cut along $(D \sqcup D')$ is a knotted 3-ball in some (unknotted) component of $V_2$ cut along $(D \sqcup D')$. See Figure \ref{fig:knottedball}.
\end{enumerate}
\end{lem}

\smallskip

\section{Winding number zero satellites}\label{s:pfs}
In this section we prove the results stated in the introduction.

Given a link $J\subset S^3$,  a \emph{ maximal system of companion tori} of $J$
is  a finite collection of tori  $\T$  in $M_J= \overline{S^3-\eta(J)}$  with the properties that the tori in $\T$ are essential, no two tori in $\T$ are parallel in $M_J$, and each component of $M_J$ cut along $\T$ is atoroidal.
By Haken's Finiteness Theorem \cite[Lemma 13.2]{hempel}  such a collection exists.

Recall that an essential annulus  in a 3-manifold $M$ with boundary is a properly embedded annulus that is 
incompressible and cannot be homotoped to lie in $\partial M$. We need the following known result.

\begin{lem} \label{lem:annuli} Let $C$ be a knot such that the complement  $M_C=\overline{S^3-\eta(C)}$
contains essential annuli. Then $C$ is either a cable or a composite knot.
\end{lem}
\begin{proof}  Let $\T$ be a maximal system of companion tori for $M_C$
and let $N$ be the component of  $M_C$ cut along $\T$ that contains  $\partial \eta(C)$.
Since $N$ is an atoroidal manifold containing essential annuli it is a Seifert fibered space.
By \cite[Lemma V1.3.4]{jacoshalen} there are three possibilities:
\begin{enumerate}
\item $\partial N$ has one component and  $M_C$ is the exterior of a torus knot.
\item $ \partial N$ has two components and $N$ is the exterior of a satellite of a torus knot in a solid torus.
\item $\partial N$ has three components and $D$ is a product of a disk with two holes with a circle.
\end{enumerate}
In cases (1), (2), $C$ is a cable knot while in case (3) $C$ is a composite knot.
\end{proof}

 Recall that $\mathbb{K}$ is the class of knots which do not admit cosmetic generalized crossing changes.

\begin{mainthm}
Let $C$ be a non-trivial,  prime,  non-cable knot and let $V'$ be a standardly embedded solid torus in $S^3$.  Let $K'\in \K$  and suppose that $K'$ is embedded in $V'$ so that it is geometrically essential and such that  $w(K',V')=0$.  Then any knot $K$ that is a satellite of $C$ with pattern $(V',  K')$  admits no  cosmetic generalized crossing changes of any order.
That is $K$ satisfies the nugatory crossing conjecture.
\end{mainthm}

\begin{proof}
Let $(V',K')$ be as in the statement of the theorem and consider the satellite map $f:(V',K') \to (V,K)$ with $\textrm{core}(V)=C$.  Suppose that $K$ admits an order-$q$ cosmetic crossing change, and let $D$ be the corresponding crossing disk with $L = \partial D$.   Let $T = \partial V$. 

 Since $C$ is a prime, non-cable knot, by Lemma \ref{lem:annuli},
there are no essential annuli in $\overline{S^3 - V}$. Hence, by Lemma \ref{lem:LinV}, we may assume $D \subset V$, so $T$ is also a companion torus for the satellite link $K \cup L$.
The link $K' \cup L'=f^{-1}(K\cup L)$ is a pattern for $K \cup L$ with the satellite map $f: (V', K', L') \to (V,K,L)$ as above.  We will show that $L'$ is an order-$q$ cosmetic crossing circle for $K'$, which is a contradiction since $K' \in \K$.  

Since $L$ is cosmetic, $M = M_{K \cup L}$ is irreducible.  
Consider a maximal system of companion tori $\T$ for
$M$.  A torus $F \in \T$ is called \emph{innermost with respect to $K$} if $M$ cut along $\T$ has a component $N$ such that $\partial N$ contains $\partial \eta (K)$ and a copy of $F$. 

\vskip 0.06in

{\it{{Claim.}}} 
Let $F\in \T$ be an  essential torus in $W=\overline{ V -  \eta(K \cup L)}$. If $F$ becomes inessential in one of 
$\overline{V-  \eta(K)}$ or $\overline{V-  \eta(K(q))}$, then it becomes inessential in both of them.
Furthermore,
$F$ 
bounds a solid torus $\widehat{F}\subset V$ with $L$ and $K$ contained in $\widehat{F}$. In particular, if $T$ is innermost with respect to $K$,
then $W$ is atoroidal.

\vskip 0.06in

{\it Proof of claim.} 
Let   $F \subset W$ be an essential torus and let $\widehat{F}$  be a solid torus bounded by $F$ in $S^3$.
Suppose, without loss of generality, that $F$ becomes inessential in
$N = \overline{V-  \eta(K)}$.
That is, in $N$ either $F$ becomes parallel to a component of 
$\partial N$ or it becomes compressible.
If $F$ is parallel to a component of 
$\partial N$, then $\widehat{F} \subset V$.
If $F$, on the other hand, becomes compressible in $N$, then we may choose a compressing disk and compress $F$ along it to obtain 2-sphere $\Sigma$. Since $N$ is irreducible, $\Sigma$ bounds a 3-ball
$B\subset N$. Hence, we have argued that an essential torus $F$ in $W$ either bounds a solid torus $\widehat{F} \subset V$ or else $F$ bounds a 
space $X\subset V$ obtained by drilling out a 1-handle from a ball $B\subset N$. Note that $X=S^3- \widehat{F}$.
We treat these two cases separately.

\vskip 0.06in

{\it Case 1.} 
$F$ bounds a 
space $X\subset V$ obtained by drilling out a 1-handle from a ball $B\subset N$.  
Then $K$ is disjoint from $X$. Furthermore, since $F$ compresses in $N$, $L$ must meet a disk that can be used
to compress $F$ down to $\Sigma$ in $N$.
 It follows that $K\cup L$  is disjoint from $X$. 
See Figure \ref{fig:figureA} for an example.

\begin{figure}
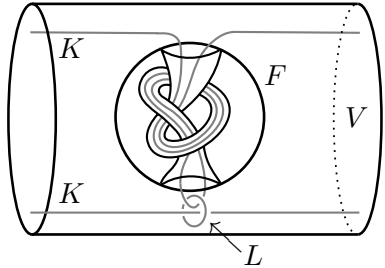
\caption{The knotted torus $F$ from Case 1.  A generalized crossing change of order $q$
will make $F$ essential in the complement of $K(q)$. }\label{fig:figureA}
\end{figure}

The torus $F$ is essential in $W$ and it 
becomes compressible in $N$, which is the 3-manifold obtained 
from $W$ by filling $\partial \eta(L)$
with slope $0$. By Gabai's result \cite[Corollary 2.4]{gabai}, $F$ must remain essential when we fill $\partial \eta(L)$
with any other slope, in particular when we fill with slope  $(-1/q)$.
Thus $F$ is essential in  $\overline{V- \eta(K(q))}$.
Since $T=\partial V$ is essential in $\overline{S^3- \eta(K(q))}$, it follows that
$F$ is essential in  $\overline{S^3- \eta(K(q))}$. On the other hand, since $F$ compresses 
in $N$, it compresses
 $\overline{S^3- \eta(K)}$. 
This is, however, impossible by Lemma \ref{lem:KinV2}.  Therefore this case will not happen.
\vskip 0.04in

{\it Case 2.} $F$ bounds a solid torus $\widehat{F}\subset V$.
As discussed above, in $N$ the torus  $F$ becomes parallel to a component of $\partial N$.
Since $F$ is essential in $W$ and $M$, it can't become parallel to $T$ in $N$.
Thus  $F$ is parallel to  $\partial \eta(K)$ in $N$.
Now we have two possibilities:
 
(i) $K \subset V - \widehat{F}$ and, since $F$ is incompressible, $L \subset \widehat{F}$;  or 

(ii) $L, K \subset \widehat{F}$.

First suppose that $K \subset V - \widehat{F}$. If $\widehat{F}$ is knotted, then either $L$ is the core of $\widehat{F}$ or $L$ is a satellite knot with companion torus $\partial \widehat{F}$.  This contradicts the fact that $L$ is unknotted.  Hence $F$ is an unknotted torus.  By definition, $L$ bounds a crossing disk $D$.  Since $D$ meets $K$ twice, $D \cap \textrm{ext}(\widehat{F}) \neq \emptyset$.  We may assume that $D$ has been isotoped (rel boundary) to minimize the number of components in $D \cap F$.  Since an innermost component of $D - (F \cap D)$ is a disk and $L$ is essential in the unknotted solid torus $\widehat{F}$, $D \cap F$ consists of standard longitudes on the unknotted torus $F$.  Hence $D \cap \textrm{ext}(\widehat{F})$ consists of either one disk which meets $K$ twice, or two disks which each meet $K$ once.  In the first case, $L$ is isotopic to the core of $\widehat{F}$, which contradicts $F$ being essential in $W$.  In the latter case, the linking number $\textrm{lk}(K,\widehat{F})= \pm 1$.  So $K$ can be considered as a connect sum $K \# U$, and the crossing change at $L$ takes place in the summand $U$.  
(See Figure \ref{fig:unknottedF}.)    Since $K'$ is geometrically essential in $V'$ there is no essential annulus in $V'- \eta(K')$
whose boundary consists of meridional curves of $\partial \eta(K')$. It follows that $U$ is the unknot.
The unknot does not admit cosmetic crossing changes of any order by \cite{schar-thom}, so $K(q) \cong K \# K'$ where $K' \not\cong U$.  This contradicts the fact that $K(q) \cong K$.  

\begin{figure}
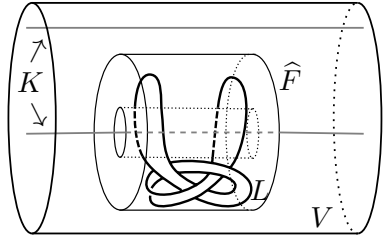
\caption{A portion of the solid torus $V$ containing the unknotted solid torus $\widehat{F}$ from Case 2. }\label{fig:unknottedF}
\end{figure}

Suppose, for a moment, that the given torus  $T$ is innermost with respect to $K$, in the sense that
 if $M$ cut along $\T$ has a component $N$ such that $\partial N$ contains $\partial \eta (K)$ and a copy of $T$.
Then possibility (ii) cannot happen, for then $F$ would be a torus in $\T$ 
that is ``closer" to $K$ than $T$. Thus, if $T$ is innermost, we can't have any essential tori in $W$.

To continue with the proof of the claim,
suppose that $L, K \subset \widehat{F}$. Then $F$ is a companion torus of the link $K\cup L$. Since $F$ becomes inessential in $N$, $F$ cannot be a companion torus of $K(0)$; in fact $K(0)$ must be the core of $\widehat{F}\subset V$ in $N$.
By Lemma \ref{lem:KinV},  $K(q)$  also becomes the core of $\widehat{F}\subset V$ in $\overline{V- \eta (K(q))}$. 
This finishes the proof of the claim.
\vskip 0.07in

To continue with the proof of the theorem, we will call a torus $F \in \T$  \emph{admissible} if it lies in  $W=\overline{ V -  \eta(K \cup L)}$ and $F$ remains essential in  $N= \overline{V- \eta(K)}$.

If $K(q)$ is not geometrically essential in $V$ 
then, by Lemma \ref{lem:KinV},  the same must be true for $K$.
But this contradicts $V$ being a companion for $K$, so $K(q)$ must be essential in $V$.  Hence $T$ is essential in both $N$ and 
$\overline{V-  \eta(K(q))}$. 

If $W$ does not contain any admissible tori then we work with the companion torus $T$. Otherwise we will replace $T$
with an admissible torus, $F\subset W$ that is innermost 
 with respect to that property. 
That is, $\overline{ \widehat{F}- \eta{(K\cup L})}$ does not contain an essential torus that remains
essential in the complement of $K$ or $K(q)$. We will argue that this  innermost companion torus  has to remain invariant
under an orientation-preserving homeomorphism of $S^3$ that brings $K$ to $K(q)$.
For simplicity of notation we will still denote this torus by $T$ and the corresponding solid torus by $V$.

Since $L$ is cosmetic, there is an ambient isotopy $\psi: S^3 \to S^3$ taking $K(q)$ to $K(0)$ such that $V$ and $\psi(V)$ are both knotted solid tori containing $K(0) = \psi(K(q)) \subset S^3$.  Then,  Lemma \ref{thm:motegi} applies to conclude that there is an ambient  isotopy $\phi: S^3 \to S^3$,
 fixing $K(0)$, such that, if we let $\Phi = (\phi \circ \psi)$,
  one of the following holds:
 \begin{enumerate}
\item  $\Phi(T) \cap T = \emptyset$.
\item There exist disjoint meridian disks $D$ and $D'$ for both $V$ and $\Phi(V)$ such that some component of $V$ cut along $(D \sqcup D')$ is a knotted 3-ball, say $B_V$, in some unknotted  component of $ \Phi(V)$ cut along $(D \sqcup D')$.  (See, for example, Figure \ref{fig:knottedball}.)
\end{enumerate}

\begin{figure}
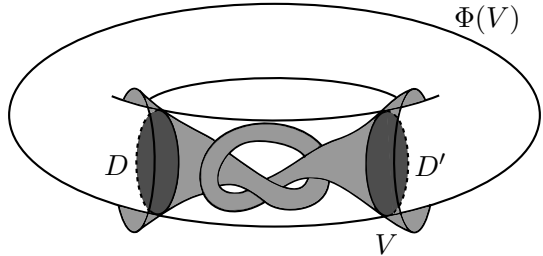
\caption{A knotted 3-ball created by the meridian disks $D$ and $D'$.}\label{fig:knottedball}
\end{figure}

Suppose we  have option (2):  Let $B'$ denote the component of $ \Phi(V)$ cut along $(D \sqcup D')$ that doesn't contain $B_V$.
The boundary of the solid torus $B'\cup B_V$, pushed slightly inside $V$, gives an essential torus in 
$\overline{V-\eta(K\cup L)}$ which remains essential  in the complement of $K$ or $K(q)$.
This, however, contradicts the assumption that $T=\partial V$ is an innermost  admissible torus. Thus option (2) cannot happen
and we conclude $\Phi(T) \cap T = \emptyset$.
Since $\T$ is a maximal system of companion tori for $M_{K\cup L}$, and  $T$ is innermost admissible,
we conclude
that $T$ and $\Phi (T)$ are parallel in $M_K$.  So, after an isotopy which fixes $K(0) \subset S^3$, we may assume that $\Phi(V) = V$.

Let $h = (f^{-1} \circ \Phi \circ f):V' \to V'$.  Then $h$ maps $K'(q)$ to $K'(0)$, 
and takes the longitude of $V'$ to itself.
The knots $K'(q)$ and $K'(0)$ are isotopic in $S^3$.  So either $L'$ gives an order-$q$ cosmetic generalized crossing change for the pattern knot $K'$, or $L'$ is a cosmetic crossing circle for $K'$. Since $K' \in \K$,  $L'$ has to be nugatory. By Lemma \ref{lem:nugasat}, $L$ is nugatory for $K$, which contradicts our assumption that $L$ is cosmetic.

\end{proof}

We now  restate and prove Corollary \ref{cor:Wd} from Section 1.

\begin{maincor}
Let $K$ be a prime knot that is not a torus knot or a cable knot. 
Then no Whitehead double of $K$ admits a cosmetic generalized crossing change of any order.
\end{maincor}

\begin{proof}
All Whitehead doubles admit a pattern $(V',U)$ where $U$ is the unknot, $w(U,V')=0$, and $U$ intersects every meridian disc of $V'$  twice. Since $U\in \K$, if $K$ is non-trivial,  the result follows immediately from Theorem \ref{thm:pattern1}.
Now suppose that $K$ is also the  unknot. Then any Whitehead double of $K$ is a \emph{twist knot}.
These knots are known to be 2-bridge knots and the conclusion follows from \cite{torisu}.  Alternatively, the conclusion
also follows from the methods of \cite{balm} since twist knots have genus one and by \cite{lyon},
admit unique minimal genus Seifert surface.

\end{proof}

A carefull observation of the proof of Theorem \ref{thm:pattern1}, reveals that the hypotheses that
$C$ be a prime,  non-cable knot  and that $w(K',V')=0$ allow us, using  Lemma 
\ref{lem:LinV}, to isotope a potential cosmetic crossing disk inside the companion solid torus
$V$. Once this is done, the above hypotheses are not used again.  In other words,
the arguments in the proof of Theorem \ref{thm:pattern1} show that if $K'\in \K$ then
no crossing circle inside $V$ can support a cosmetic crossing change. In fact we have the following.

\begin{thm}
Let $C\in \K$ be a non-trivial knot and let $V'$ be a standardly embedded solid torus in $S^3$. Let $K'\in \K$ and suppose that $K'$ is embedded in $V'$ so that it is geometrically essential. Then any knot $K$ that is a satellite of $C$ with pattern $(V', K')$ admits no cosmetic generalized crossing changes supported on crossing circles disjoint from
the companion torus $T=\partial V$.

\end{thm}

\begin{proof}
Let $(V',K')$ be as in the statement of the theorem and consider the satellite map $f:(V',K') \to (V,K)$ with $\textrm{core}(V)=C$. Suppose that $K$ admits an order-$q$ cosmetic crossing change, and let $D$ be the corresponding crossing disk with $L = \partial D$, such that $L\cap T=\emptyset$.

After isotopy we may assume that each component of $D\cap T$ bounds a disk in $D$ whose interior intersects
$K$.

First suppose that $D\cap T$ contains a component bounding a disk $E\subset D$
whose interior is intersected exactly once by $K$. The disk $E$ is a meridian disk of $V$.
It follows that $K$ is a composite knot, $K=K'\# C$, and $T$ is the follow-swallow torus.
Furthermore, the cosmetic order-$q$ crossing must occur on the companion $C$. However,
this contradicts the fact that $C\in \K$. Thus this case will not happen.

Assume, therefore, that every component of $D\cap T$ bounds a disk on $D$ whose interior is pierced exactly
twice by $K$. It follows that every component of $D\cap T$ is parallel to $\partial D$ on $D$. Thus we may slide $T$ off of $D$ to assure that $D \subset V$.
Now $T$ is also a companion torus for the satellite link $K \cup L$.
The link $K' \cup L'=f^{-1}(K\cup L)$ is a pattern for $K \cup L$ with the satellite map $f: (V', K', L') \to (V,K,L)$ as above.
The situation is exactly as before the statement of the Claim in the proof of Theorem \ref{thm:pattern1}.
The argument therein applies to show
that $L'$ is an order-$q$ cosmetic crossing circle for $K'$, which is a contradiction since $K' \in \K$.

\end{proof}

\bibliographystyle{amsplain}
\bibliography{references}

\end{document}